\documentclass{article}
\usepackage[a4paper, total={6in, 10in}]{geometry}
\usepackage{amsmath,amsthm,amssymb,mathrsfs,amstext, titlesec,enumitem, comment, graphicx, color, xcolor, mathabx,textcomp}

\usepackage{mathdots}
\usepackage{MnSymbol}
\usepackage{hyperref}
\usepackage{xpatch}
\usepackage{indentfirst} 

\usepackage{srcltx}
\usepackage{epsfig}
\usepackage{url}

\makeatletter
\def\Ddots{\mathinner{\mkern1mu\raise\p@
\vbox{\kern7\p@\hbox{.}}\mkern2mu
\raise4\p@\hbox{.}\mkern2mu\raise7\p@\hbox{.}\mkern1mu}}
\newcommand{\ostarr}{\mathbin{\mathpalette\make@circled\star}}
\newcommand{\make@circled}[2]{%
  \ooalign{$\m@th#1\smallbigcirc{#1}$\cr\hidewidth$\m@th#1#2$\hidewidth\cr}%
}
\newcommand{\smallbigcirc}[1]{%
  \vcenter{\hbox{\scalebox{0.77778}{$\m@th#1\bigcirc$}}}%
}
\makeatother

\titleformat{\section}
{\normalfont\Large\bfseries}{\thesection}{1em}{}
\titleformat*{\subsection}{\Large\bfseries}
\titleformat*{\subsubsection}{\large\bfseries}
\titleformat*{\paragraph}{\large\bfseries}
\titleformat*{\subparagraph}{\large\bfseries}

\makeatother
\theoremstyle{plain}
\newtheorem{thm}{Theorem}[section]
\newtheorem{prop}[thm]{Proposition}
\newtheorem{lem}[thm]{Lemma}
\newtheorem{cor}[thm]{Corollary}
\theoremstyle{definition}
\newtheorem{fact}[thm]{Fact}
\newtheorem{defn}[thm]{Definition}
\newtheorem{rem}[thm]{Remark}

\newtheorem{question}[thm]{Question}
\newtheorem{con}[thm]{Conjecture}
\newtheorem{example}[thm]{Example}

\theoremstyle{plain}
\newcommand{\thistheoremname}{}
\newtheorem*{genericthm*}{\thistheoremname}
\newenvironment{namedthm*}[1]
  {\renewcommand{\thistheoremname}{#1}%
   \begin{genericthm*}}
  {\end{genericthm*}}

\newcommand{\N}{\mathbb{N}}
\newcommand{\Z}{\mathbb{Z}}

\newcommand{\R}{\mathbb{R}}

\newcommand{\bN}{\beta\mathbb{N}}

\date{\vspace{-5ex}}

\begin{document}

\title{Exponentiations of ultrafilters}

\author{Lorenzo Luperi Baglini \\ 
\textit{lorenzo.luperi@unimi.it}
\footnote{Dipartimento di Matematica, Universit\`{a} di Milano, Via Saldini 50, 20133 Milano, Italy. \url{lorenzo.luperi@unimi.it}. Supported by PRIN 2022 ``Logical methods in combinatorics'', 2022BXH4R5, MIUR (Italian Ministry of University and Research).}}

\maketitle
\begin{abstract}
In recent years, several problems regarding the partition regularity of exponential configurations have been studied in the literature, in some cases using the properties of specific ultrafilters. In this paper, we start to lay down the foundations of a general theory of exponentiations of ultrafilters, with a particular focus on their combinatorial properties and the existence of idempotents.
\end{abstract}


\section{Introduction}

A large class of problems in arithmetic Ramsey theory regards the so-called partition regular configurations and equations on $\N=\{1,2,\ldots\}$. We let $\omega=\N\cup\{0\}$.

\begin{defn} Let $f_{1}\left(x_{1},\ldots,x_{n}\right),\ldots,f_{m}\left(x_{1},\ldots,x_{n}\right):\N^{n}\rightarrow \N$. We say that the configuration $\{f_{1}\left(x_{1},\ldots,x_{n}\right),\ldots,f_{m}\left(x_{1},\ldots,x_{n}\right)\}$ is partition regular on $\N$ (PR from now on) if for all finite partitions $\N=A_{1}\cup\ldots\cup A_{k}$ there are $i\leq k$ and $x_{1},\ldots,x_{n}\in\N$ such that for all $j\leq m$ $f_{j}\left(x_{1},\ldots,x_{n}\right)\in A_{i}$.

Analogously, if $f_{1}\left(x_{1},\ldots,x_{n}\right),\ldots,f_{m}\left(x_{1},\ldots,x_{n}\right):\N^{n}\rightarrow \omega$ and the system $\sigma\left(x_{1},\ldots,x_{n}\right)$ is defined by \[\begin{cases} f_{1}\left(x_{1},\ldots,x_{n}\right)=0,\\ \ \ \ \ \ \ \ \ \ \vdots \\ f_{m}\left(x_{1},\ldots,x_{n}\right)=0\end{cases}\] we say that $\sigma\left(x_{1},\ldots,x_{n}\right)$ is PR if for every finite partition $\N=A_{1}\cup\ldots\cup A_{k}$, there are $i\leq k$ and $x_{1},\ldots,x_{n}\in A_{i}$ such that  $\sigma\left(x_{1},\ldots,x_{n}\right)=\vec{0}$.\end{defn}

In \cite{key-2}, Schur proved one of the first and most important results in rithmetic Ramsey theory. 
\begin{thm}[Additive Schur's Theorem]\label{add schur} The configuration $\{ x,y,x+y\}$ is PR. 
\end{thm}

Equivalently, the additive Schur's Theorem states that the equation $x+y=z$ is PR. 

Applying transformations to partitions allows to easily deduce new results by just considering the shape of transformed partition regular configurations. A useful trick, that sometimes allows turning additive structures in multiplicative structures, and multiplicative structures in exponential ones, is that of considering $\log_{2}$-base partitions: for $B\subseteq\N$, let $2^{-B}=\{n\in\N\mid 2^{n}\in B\}$ and, given a partition $\N=A_{1}\cup\ldots\cup A_{k}$, let the associated $\log_2$-base partition be given by $\N=2^{-A_{1}}\cup\ldots\cup 2^{-A_{k}}$. For example, given any finite partition of $\N$, in the $\log_2$-base partition we find monochromatic $x,y,z$ with $x+y=z$ by Theorem \ref{add schur}; hence, $2^{x},2^{y},2^{z}=2^{x+y}=2^{x}\cdot 2^{y}$ were monochromatic with respect to the original partition. This proves the multiplicative version of Schur.

\begin{thm}[Multiplicative Schur's Theorem] The configuration $\{x,y,x\cdot y\}$ is PR. \end{thm}

Equivalently, the multiplicative Schur Theorem states that the equation $xy=z$ is PR. 

Recently, in \cite{key-3}, the question about the partition regularity of $\{x,y,x^{y}\}$, which we will call an exponential triple, was posed and partially answered, by using ultrafilters, by Sisto; its full solution was first given in \cite{key-4} by Sahasrabudhe who, among other generalizations, proved the exponential Schur's Theorem\footnote{At the best of our knowledge, this theorem has not been called explicitly like this in the literature before.}:

\begin{thm}[Exponential Schur's Theorem] The configuration $\{x,y,x^{y}\}$ is PR. \end{thm}

Again, the above theorem can be rephrased in terms of equations: it states that $x^{y}=z$ is PR. Sahasrabudhe's proof used the compactness property on the convergence of subsequences in the space of $k$-partitions of $\N$ and repeated applications of van der Waerden's Theorem:

\begin{thm}[Van der Waerden's Theorem] For all $k\in\N$ the configuration $\{x,x+y,\ldots,x+ky\}$ is PR. \end{thm}

One of the key ingredients in Sahasrabudhe's proof was the use of the function $f(n,m):=2^{n}m$, combined with associated $\log_{2}$-base partitions. Starting from observations about this function, in \cite[Theorem 2.1]{key-5} Di Nasso and Ragosta gave a very simple and elegant ultrafilters proof of the partition regularity of the configuration $\{x,y,2^{x}y\}$, from which the exponential Schur's theorem can be immediately deduced by the same $\log_{2}$-partition trick that allows proving the multiplicative Schur's Theorem from the additive one (see \cite[Corollary 2.3]{key-5}). Their proof was based on tensor products of ultrafilters\footnote{We will recall some basic facts and definitions regarding ultrafilters in Section \ref{prel}, but we assume the reader to have a basic knowledge of the main properties of $\beta\N$; we refer to the monograph \cite{key-22} for a detailed presentation of ultrafilters and their applications to combinatorics.} and the extension to $\left(\beta\N\right)^{2}$ of the same function $f(n,m)=2^{n}m$ used by Sahasrabudhe. Di Nasso and Ragosta used ultrafilters methods also to study the partition regularity of infinite exponential configurations in \cite{key-6}. As it is often the case with partition regularity problems, the ultrafilters proof allows for several generalizations: this is a common situation,  as proven by an enormous literature on the subject (instead of citing here hundreds of papers, we refer to \cite{key-22} and its bibliography). 

Our main goal is to start a systematic study of which ultrafilters are related to exponential triples, and what other combinatorial properties such ultrafilters have. This will be done by introducing, in Section \ref{expul}, two different exponentiations of ultrafilters, named $E_{1},E_{2}$, and by studying their algebraic properties and their connections with the exponential Schur's Theorem. In Section \ref{tr} we study the existence of idempotents for $E_{1},E_{2}$, and prove that $E_{1}$ does not admit nonprincipal idempotents, whilst $E_{2}$ does not admit nonprincipal idempotents that are simultaneously multiplicative idempotents.  The novelty of most of our results generates a large number of open questions and future possible directions. We conclude the paper with a discussion of many of these in Section \ref{questions}.

\section{Exponentiations of ultrafilters}  \label{expul}

\subsection{Preliminaries}\label{prel}

As ultrafilters are the key notion of this paper, we assume the reader to have a basic knowledge of the algebra of $\beta\N$. We just briefly recall a few definitions, mostly so that we can fix some notations.

Sums and products of ultrafilters will be denoted by $\oplus,\odot$ respectively. They are related to tensor products of ultrafilters which, over $\N$, are defined as follows:

\begin{defn} Given $p,q\in\beta\N$, the tensor product $p\otimes q$ of $p,q$ is the ultrafilter on $\beta(\N^{2})$ defined as follows: for all $A\subseteq\N^{2}$ 
\[A\in p\otimes q\Leftrightarrow\{a\in\N\mid \{b\in\N\mid (a,b)\in A\}\in q\}\in p.\]\end{defn}

Moreover, we will need to talk about extensions of functions to $\beta\N$: for $f:\N^{n}\rightarrow \N$, we will denote by $\widehat{f}$ its unique continuous extension $\widehat{f}:\beta(\N^{n})\rightarrow \beta\N$. As usual, to simplify notations we identify each natural number $n\in\N$ with the principal idempotent it generates, so that $\N\subseteq\beta\N$, and we convene that, for $a\in\N, p\in\beta\N$, we will write $ap$ instead of $a\odot p$. As in Proposition \ref{alg exp} we will use ultrafilters over $\R^{+}_{d}$, we recall that $\beta\R^{+}_{d}$ is the Stone-\v{C}ech compactification of $\R^{+}:=\{x\in \R\mid x>0\}$ endowed with the discrete topology. We adopt for $\beta\R^{+}_{d}$ the analogous notational conventions we are using for $\beta\N$.  

As well-known, ultrafilters are often a fundamental tool to study PR problems. For configurations, the connection between ultrafilters and partition regularity reads as follows (see e.g. \cite[Theorem 5.7]{key-22}):

\begin{fact} Let $f_{1}\left(x_{1},\ldots,x_{n}\right),\ldots,f_{m}\left(x_{1},\ldots,x_{n}\right):\N^{n}\rightarrow \N$. The following are equivalent:
\begin{enumerate}
\item $\{f_{1}\left(x_{1},\ldots,x_{n}\right),\ldots,f_{m}\left(x_{1},\ldots,x_{n}\right)\}$ is PR;
\item $\exists p\in\beta\N\ \forall A\in p \ \exists x_{1},\ldots,x_{n}\in \N\ f_{1}\left(x_{1},\ldots,x_{n}\right),\ldots,f_{m}\left(x_{1},\ldots,x_{n}\right)\in A$.
\end{enumerate}
\end{fact}

Analogously, for systems the following holds:

\begin{fact} Let $f_{1}\left(x_{1},\ldots,x_{n}\right),\ldots,f_{m}\left(x_{1},\ldots,x_{n}\right):\N^{n}\rightarrow \omega$, and let the system $\sigma\left(x_{1},\ldots,x_{n}\right)$ be defined by \[\begin{cases} f_{1}\left(x_{1},\ldots,x_{n}\right)=0,\\ \ \ \ \ \ \ \ \ \ \vdots \\ f_{m}\left(x_{1},\ldots,x_{n}\right)=0.\end{cases}\] The following are equivalent:
\begin{enumerate}
\item $\sigma\left(x_{1},\ldots,x_{n}\right)=\vec{0}$ is PR;
\item $\exists p\in\beta\N\ \forall A\in p \ \exists x_{1},\ldots,x_{n}\in A\ \sigma\left(x_{1},\ldots,x_{n}\right)=\vec{0}$.
\end{enumerate}
\end{fact}

The above facts lead to the following definitions:

\begin{defn} Let $p\in\beta\N$ and let $f_{1}\left(x_{1},\ldots,x_{n}\right),\ldots,f_{m}\left(x_{1},\ldots,x_{n}\right):\N^{n}\rightarrow \N$. We say that $p$ witnesses the partition regularity of the configuration $\{f_{1}\left(x_{1},\ldots,x_{n}\right),\ldots,f_{m}\left(x_{1},\ldots,x_{n}\right)\}$ if for all $A\in p$ there exist $x_{1},\ldots,x_{n}\in \N$ such that $f\left(x_{1},\ldots,x_{n}\right)\in A$. In this case, we write $p\models \{f_{1}\left(x_{1},\ldots,x_{n}\right),\ldots,f_{m}\left(x_{1},\ldots,x_{n}\right)\}$. Analogously, if the system $\sigma\left(x_{1},\ldots,x_{n}\right)$ is defined by \[\begin{cases} f_{1}\left(x_{1},\ldots,x_{n}\right)=0,\\ \ \ \ \ \ \ \ \ \ \vdots \\ f_{m}\left(x_{1},\ldots,x_{n}\right)=0\end{cases}\] we say that $p$ witnesses the partition regularity of $\sigma\left(x_{1},\ldots,x_{n}\right)=\vec{0}$ if $\forall A\in p \ \exists x_{1},\ldots,x_{n}\in A\ \sigma\left(x_{1},\ldots,x_{n}\right)=\vec{0}$. In this case, we write $p\models \sigma\left(x_{1},\dots,x_{n}\right)=\vec{0}$.\end{defn}

In the next sections we will use a few specific results from other papers. For the readers' easiness, we recall them here. The first states that when the PR of different systems of equations are witnessed by the same ultrafilter, these configurations can be mixed.

\begin{lem}{\cite[Lemma 2.1]{advances}}\label{old1} Let $p\in\beta\N$, and let $f\left(x_{1},\ldots,x_{n}\right),g\left(y_{1},\ldots,y_{m}\right)$ be two functions such that $p\models \left(f\left(x_{1},\ldots,x_{n}\right)=0\right)\wedge\left(g\left(y_{1},\ldots,y_{m}\right)=0\right)$. Then $p\models\left(f\left(x_{1},\ldots,x_{n}\right)=0\right)\wedge\left(g\left(y_{1},\ldots,y_{m}\right)=0\right)\wedge\left(x_{1}=y_{1}\right)$.\end{lem}

The second result that we will need is \cite[Theorem 1.6]{GP}, which states that multiplicatively piecewise syndetic sets contains some exponential structure.

\begin{thm}{\cite[Theorem 1.6]{GP}}\label{GosPa} Let $n\in\mathbb{N}, n>1$ and let $A\subseteq\N$ be multiplicatively piecewise syndetic. Then there exists $x,y\in A$ such that $xn^{y}\in A$. \end{thm}

The third and fourth result we need are \cite[Theorem A]{mal} and (the restriction to $\N$ of) \cite[Theorem 13.14]{key-22}\footnote{That we slightly rephrase to adapt to our notations.}, that states that certain ultrafilter equations admit no solution.

\begin{thm}{\cite[Theorem A]{mal}}\label{TMal} Let $p\in\beta\N\setminus \N$, $a, b\in\N$ and $u,v\in\beta\N$. Suppose that $u\oplus ap = v \oplus bp$. Then $a = b$.\end{thm}

\begin{thm}{\cite[Theorem 13.14]{key-22}}\label{HSsp} Let $p, q, r, s\in\beta\N\setminus\N$. If $|\{a\in\N\mid a\Z\in r\}|=|\N|$, then $q\oplus p\neq s\odot r$.\end{thm}

Finally, we will use the non-PR of combined sums and exponentiations stated in \cite[Theorem 6]{key-4}.

\begin{thm}{\cite[Theorem 6]{key-4}}\label{Sahasr} The configuration $\{x,y,x^{y},a,b,a+b\}$ is not PR. \end{thm}

\subsection{Two extensions of exponentiation to ultrafilters}\label{extexp}

As it is well known, there are two natural ways to extend the sum from $\N$ to $\beta\N$: the usual one $(p,q)\rightarrow p\oplus q$ defined by letting, for every $A\in\N$,
\[A\in p\oplus q\Leftrightarrow \{n\in\N\mid \{m\in\N\mid n+m\in A\}\in q\}\in p,\]
and the reverse sum $p\tilde{\oplus} q:= q\oplus p$; if $S(n,m):=n+m$ denotes the sum of natural numbers, then $p\oplus q = \widehat{S}(p\otimes q)$ and $q\oplus p=\widehat{S}(q\otimes p)$. Whilst different, $\oplus$ and $\tilde{\oplus}$ have the same properties, modulo properly translating their meaning by switching ``left'' with ``right'' notions. This is a consequence of the sum being commutative on $\N$. 

Exponentiation, on the other hand, is not commutative nor associative. Hence, when extending it to $\beta\N$, we have at least two natural and truly distinct possible definitions to consider\footnote{Not including those that could be obtained trivially by reversing the order of the variables in the operations.}.

\begin{defn}\label{expdef} Given $p,q\in\beta\N$ we let 
\[E_{1}(p,q):=\{A\subseteq\N\mid \{n\in\N\mid \{m\in\N\mid n^{m}\in A\}\in q\}\in p\},\]
and
\[E_{2}(p,q):=\{A\subseteq\N\mid \{n\in\N\mid \{m\in\N\mid m^{n}\in A\}\in q\}\in p\}.\]
\end{defn}

Equivalently, if $e_{1},e_{2}:\N^{2}\rightarrow \N$ are the exponentiations $e_{1}(n,m)=n^{m}, e_{2}(n,m)=m^{n}$, then $E_{1}(p,q)=\widehat{e_{1}}(p\otimes q)$ and $E_{2}(p,q)=\widehat{e_{2}}(p\otimes q)$.

For $n,m\in\N$, trivially we have that $E_{1}(n,m)=E_{2}(m,n)=n^{m}$; more in general, for $n\in\N, p\in\beta\N$ we have that $E_{1}(n,p)=E_{2}(p,n)$ and $E_{1}(p,n)=E_{2}(n,p)$ as, for all $A\subseteq\N$
\begin{gather*}A\in E_{1}(n,p)\Leftrightarrow \{m\in \N\mid n^{m}\in A\}\in p\Leftrightarrow A\in E_{2}(p,n),\\
A\in E_{1}(p,n)\Leftrightarrow \{m\in \N\mid m^{n}\in A\}\in p\Leftrightarrow A\in E_{2}(n,p).
\end{gather*}

{\bf Notation:} we will use the simplified notation $n^{p}$ to denote any of $E_{1}(n,p),E_{2}(p,n)$, and we will use the simplified notation $p^{n}$ to denote any of $E_{1}(p,n),E_{2}(n,p)$.

\bigskip

At first sight, it might look like the equality $E_{1}(p,q)= E_{2}(q,p)$ could be true for all $p,q\in\beta\N$; however, this is not the case. 

\begin{example} Let both $p,q$ be nonprincipal ultrafilters that contain the set $P$ of primes. Then 
\[P_{1}=\{a^{b}\mid a,b\in P, a<b\}\in E_{1}(p,q), P_{2}=\{a^{b}\mid a,b\in P, a>b\}\in E_{2}(p,q)\]
but $P_{1}\cap P_{2}=\emptyset$, so $E_{1}(p,q)\neq E_{2}(q,p)$.\end{example}

 We will provide some examples of nonprincipal ultrafilters $p,q$ satisfying $E_{1}(p,q)= E_{2}(q,p)$ after Proposition \ref{alg exp}. To prove this proposition, it is helpful to use fractional powers of ultrafilters; to make proper sense of these powers, we prefer to state and prove Proposition \ref{alg exp} for ultrafilters in $\beta\R^{+}_{d}$. We extend the exponentiations to $\beta\R_{d}$ as follows: given $p,q\in\beta\R^{+}_{d}$, we let:
\[E_{1}(p,q):=\{A\subseteq\R^{+}\mid \{x\in\R^{+}\mid \{y\in\R^{+}\mid x^{y}\in A\}\in q\}\in p\},\]
and
\[E_{2}(p,q):=\{A\subseteq\R^{+}\mid \{x\in\R^{+}\mid \{y\in\R^{+}\mid y^{x}\in A\}\in q\}\in p\}.\]
For $r\in\R^{+}$ and $p\in\beta\R^{+}_{d}$, we will keep the simplified notation $r^{p}$ to denote any of $E_{1}(r,p),E_{2}(p,r)$ (which are equal), and $p^{r}$ to denote any of $E_{1}(p,r),E_{2}(r,p)$ (which are also equal).

We start with a preliminary lemma generalizing standard laws of exponents and logarithms.

\begin{lem}\label{lem:basicproplogexp} Let $p\in\beta\R_{d}^{+}$ and let $a,b\in\R^{+}\setminus\{1\}$. The following facts hold:
\begin{enumerate}
\item $\widehat{\log_{a}}\left(b^{q}\right)=\log_{a}(b)\odot q$;
\item $\left(a^{b}\right)^{q}=a^{bq}$;
\item $\left(p^{a}\right)^{b}=p^{ab}$.
\end{enumerate}
\end{lem}

\begin{proof} We prove only $(1)$, as $(2)-(3)$ are similar. Let $A\subseteq\R^{+}$. Then
\begin{gather*} A\in \widehat{\log_{a}}\left(b^{q}\right)\Leftrightarrow \{x\in\R^{+}\mid \log_{a}(x)\in A\}\in b^{q} \Leftrightarrow\\ \{z\in\R^{+}\mid b^{z}\in \{x\in\R^{+}\mid \log_{a}(x)\in A\}\}\in q\Leftrightarrow \{z\in \R^{+}\mid \log_{a}(b^{z})\in A\}\in q\Leftrightarrow \\ \{z\in \R^{+}\mid z\cdot \log_{a}(b)\in A\}\in q\Leftrightarrow A\in \log_{a}(b)\odot q.\qedhere\end{gather*}
\end{proof}

\begin{prop}\label{alg exp} Let $p,q,r\in\beta\R_{d}^{+}\setminus\{1\}$. Let $a,b\in\R^{+}$ with $a\geq b>1$. The following properties hold:

\begin{enumerate}
\item $E_{1}\left(p^{a},q\right)=E_{1}\left(p,aq\right)$ and $E_{2}\left(q,p^{a}\right)=E_{2}\left(aq,p\right)$;
\item $a^{p}=b^{q}\Leftrightarrow \left(\log_{a}(b)\right)\odot p=q$; 
\item $p^{a}=q^{b}\Leftrightarrow p^{\frac{a}{b}}=q$; 
\item $a^{p}\odot a^{q}=a^{p\oplus q}$;
\item $E_{1}\left(p,q\odot r\right)=E_{1}\left(E_{1}(p,q),r\right)$; in particular, if $q$ is multiplicatively idempotent, $E_{1}(p,q)=E_{1}\left(E_{1}(p,q),q\right)$;
\item $E_{2}\left(p\odot q,r\right)=E_{2}\left(p,E_{2}(q,r)\right)$; in particular, if $p$ is multiplicatively idempotent, $E_{2}(p,q)=E_{2}\left(p,E_{2}(p,q)\right)$.
\end{enumerate}
\end{prop}

\begin{proof} $(1)$ Let $A\subseteq\R^{+}$. Then 
\begin{gather*}
A\in E_{1}\left(p^{a},q\right)\Leftrightarrow \{x\in\R^{+}\mid \{y\in\R^{+}\mid x^{y}\in A\}\in q\}\in p^{a}\Leftrightarrow\\ \{z\in\R^{+}\mid z^{a}\in\{x\in\R^{+}\mid\{y\in\R^{+}\mid x^{y}\in A\}\in q\}\}\in p\Leftrightarrow\\ \{z\in\R^{+}\mid \{y\in\R^{+}\mid \left(z^{a}\right)^{y}\in A\}\in q\}\in p \Leftrightarrow \{z\in\R^{+}\mid \{y\in\R^{+}\mid z^{\left(ay\right)}\in A\}\in q\}\in p\Leftrightarrow \\ \{z\in \R^{+}\mid \{t\in\R^{+}\mid z^{t}\in A\}\in aq\}\in p\Leftrightarrow A\in E_{1}(p,aq).\end{gather*}
The $E_{2}$-case is analogous.

$(2)$ If $a^{p}=b^{q}$, by Lemma \ref{lem:basicproplogexp}.(1) $p=\widehat{\log_{a}}\left(a^{p}\right)=\widehat{\log_{a}}\left(b^{q}\right)=\left(\log_{a}(b)\right)\odot q$; conversely, if $p=\left(\log_{a}(b)\right)\odot q$ then $a^{p}=a^{\left(\log_{a}(b)\right)\odot q}=\left(a^{\log_{a}(b)}\right)^{q}=b^{q}$ by Lemma \ref{lem:basicproplogexp}.(2).

$(3)$ If $p^{a}=q^{b}$, by Lemma \ref{lem:basicproplogexp}.(3) $p^{\frac{a}{b}}=\left(p^{a}\right)^{\frac{1}{b}}=\left(q^{b}\right)^{\frac{1}{b}}=q$; conversely, again by Lemma \ref{lem:basicproplogexp}.(3), if $p^{\frac{a}{b}}=q$ then $q^{b}=\left(p^{\frac{a}{b}}\right)^{b}=p^{\frac{ab}{b}}=p^{a}$.

$(4)$ Let $A\subseteq\R^{+}$. Then
\begin{gather*}
A\in a^{p}\odot a^{q}\Leftrightarrow \{x\in \R^{+}\mid \{y\in \R^{+}\mid xy\in A\}\in a^{q}\}\in a^{p}\Leftrightarrow\\ \{x\in\R^{+}\mid \{t\in \R^{+}\mid a^{t}\in\{y\in \R^{+}\mid xy\in A\}\}\in q\}\in a^{p}\Leftrightarrow\\ \{x\in\R^{+}\mid \{t\in\R^{+}\mid x\cdot a^{t}\in A\}\in q\}\in a^{p}\Leftrightarrow \{z\in \R^{+}\mid a^{z}\in \{ x\in \R^{+}\mid \{t\in\R^{+}\mid x\cdot a^{t}\in A\}\in q\}\in p\Leftrightarrow\\ \{z\in \R^{+}\mid \{t\in \R^{+}\mid a^{t+z}\in A\}\in q\}\in p\Leftrightarrow
   A\in a^{p\oplus q}. \end{gather*}

$(5)$ Let $A\subseteq\R^{+}$. Then
\begin{gather*}
A\in E_{1}(p,q\odot r)\Leftrightarrow\{x\in\R^{+}\mid \{y\in\R^{+}\mid x^{y}\in A\}\in q\odot r\}\in p\Leftrightarrow\\
\{x\in\R^{+}\mid \{t\in\R^{+}\mid \{z\in\R^{+}\mid x^{tz}\in A\}\in r\}\in q\}\in p\Leftrightarrow\\ \{x\in\R^{+}\mid \{t\in\R^{+}\mid \{z\in\R^{+}\mid \left(x^{t}\right)^{z}\in A\}\in r\}\in q\}\in p\Leftrightarrow\\\{y\in\R^{+}\mid \{z\in\R^{+}\mid s^{z}\in A\}\in r\}\in E_{1}\left(p,q\right)\Leftrightarrow
 A\in E_{1}\left(E_{1}\left(p,q\right),r\right).
\end{gather*}

$(6)$ This is analogous to $(5)$. Let $A\subseteq\R^{+}$. Then 
\begin{gather*}
A\in E_{2}\left(p,E_{2}(q,r)\right)\Leftrightarrow \{x\in\R^{+}\mid\{y\in\R^{+}\mid y^{x}\in A\}\in E_{2}(q,r)\}\in p\Leftrightarrow\\\{x\in\R^{+}\mid\{t\in\R^{+}\mid\{z\in\R^{+}\mid \left(z^{t}\right)^{x}\in A\}\in r\}\in q\}\in p\Leftrightarrow\\\{x\in\R^{+}\mid\{t\in\R^{+}\mid\{z\in\R^{+}\mid z^{tx}\in A\}\in r\}\in q\}\in p\Leftrightarrow\{s\in\R^{+}\mid\{z\in\R^{+}\mid z^{s}\in A\}\in r\}\in p\odot q.\qedhere
\end{gather*}
    \end{proof}

\begin{cor} Let $p,q\in\beta\N\setminus{1}$, $a,b\in\N\setminus\{1\}$. Then
\begin{enumerate}
\item $a^{p}=a^{q}\Leftrightarrow p=q$;
\item $p^{a}=p^{b}\Leftrightarrow a=b$.
\end{enumerate}
\end{cor}

\begin{proof} (1) This follows from Proposition \ref{alg exp}.(2).

(2) This follows from Proposition \ref{alg exp}.(3), as $p^{\frac{a}{b}}=p$ if and only if $\frac{a}{b}=1$. \end{proof}

\begin{example}\label{equalE1E2}
    Let $n\in\N\setminus\{1\}$ and $p,q\in\beta\N $ be such that $p\odot q=q\odot p$. Then by Proposition \ref{alg exp} we have that 
    \[E_{1}\left(n^{p},q\right)=n^{p\odot q}= n^{q\odot p}=E_{2}\left(q,n^{p}\right).\]
    Let us notice that the equation $p\odot q=q\odot p$ can be solved by many pairs $p,q\in\beta\N\setminus\N$, for example by letting $p=q$, or by taking $p,q$ multiplicative idempotents with $p<q$. Hence, the equation $E_{1}(u,v)=E_{2}(u,v)$ has plenty of nontrivial solutions in $\beta\N\times \beta\N$.
\end{example}

Proposition \ref{alg exp} might again suggest that many arithmetic properties of the exponentiations of natural numbers hold also for $E_{1},E_{2}$. For example, Proposition \ref{alg exp}.(4) generalizes the usual rule $a^{b}\cdot a^{c}=a^{bc}$ to the case where $b,c$ are ultrafilters. However, this same rule cannot be generalized to the case where $a$ is an ultrafilter, as we show in the next proposition.

\begin{prop}\label{neqr} Let $p\in\beta\N\setminus \N$, and let $a,b\in\N$. Then $p^{a}\odot p^{b}\neq p^{a+b}$.\end{prop}

\begin{proof} Assume, by contrast, that $p^{a}\odot p^{b}= p^{a+b}$. 

Let $F:\N\rightarrow\N$ be the function that maps a number $n$ to its largest prime divisor. Observe that, for all $p_{1},p_{2}\in\beta\N$, $\widehat{F}(E_{2}(p_{1},p_{2}))=p_{1}\odot \widehat{F}(p_{2})$. We consider two cases.

{\bfseries Case 1:} $\widehat{F}(p)=i\in\N$. In this case, consider the function $\Omega:\N\rightarrow\N$ that maps a number $n$ to the sum of the exponents in its prime factorization. Necessarily, $\widehat{\Omega}(p)\notin\N$; in fact, if by contrast $\widehat{\Omega}(p)=j\in\N$, we would have that $p\in\overline{S}$, where
\[S=\{n\in\N\mid F(n)=i, \Omega(n)=j\},\]
which is a contradiction as $p$ is nonprincipal and $S$ is finite. Hence $\widehat{\Omega}(p)\notin\N$. As $p^{a}\odot p^{b}= p^{a+b}$, we have that \[\widehat{\Omega}\left(p^{a}\odot p^{b}\right)=\widehat{\Omega}\left(p^{a}\right)\oplus \widehat{\Omega}\left(p^{b}\right)=a\widehat{\Omega}(p)\oplus b\widehat{\Omega}(p),\] 
whilst \[\widehat{\Omega}\left(p^{a+b}\right)=(a+b)\widehat{\Omega}(p).\]
We conclude as $a\widehat{\Omega}(p)\oplus b\widehat{\Omega}(p)\neq (a+b)\widehat{\Omega}(p)$ by Theorem \ref{TMal}. 

{\bfseries Case 2:} $\widehat{F}(p)\notin\N$. Let $G:\N\rightarrow\N$ be the function $n\rightarrow \max\{m\in\N\mid F(n)^{m} \ \text{divides} \ n\}$, and let $H(n):=F(n)^{G(n)}$. As $\widehat{F}(p)$ is nonprincipal, also $\widehat{H}(p)$ is nonprincipal. Notice that:

\begin{itemize}
\item $\widehat{H}\left( p^{a}\odot p^{b}\right)=b\widehat{H}(p)$: in fact, for $x\in\N$, let $D_{x}=\{y\in\N\mid F(y)> F(x)\}$. If $x\in\N$ and $q\in\overline{D_{x}}$, then $\widehat{H}(x\odot q)=\widehat{H}(q)$. Therefore, if we let $\rho_{q}:\beta\N\rightarrow\beta\N$ be the function mapping $u\in\beta\N$ into $u\odot q$, we have that $\widehat{H}\circ\rho_{q}$ is constantly equal to $\widehat{H}(q)$ on $\N$; as both these functions are continuous and $\N$ is dense in $\beta\N$, $\widehat{H}(p\odot q)=\widehat{H}(q)$. Therefore $\widehat{H}\left(p^{a}\odot p^{b}\right)=\widehat{H}\left(p^{b}\right)=b\widehat{H}(p)$;
\item $(a+b)\widehat{H}(p)=\widehat{H}\left(p^{a+b}\right)$, as for all $n\in\N$ $(a+b)H(n)=H\left(n^{a+b}\right)$.
\end{itemize}

Therefore
\[\widehat{H}\left( p^{a}\odot p^{b}\right)=b\widehat{H}(p)\neq (a+b)\widehat{H}(p)=\widehat{H}\left(p^{a+b}\right),\]
again by Theorem \ref{TMal}. \end{proof}

\begin{rem} The above result proves that the equalities $E_{1}(p,a)\odot E_{1}(p,b)= E_{1}(p,a+b), E_{2}(a,p)\odot E_{2}(b,p)= E_{2}(a+b,p)$ does not hold when $a,b\in\N,p\in\beta\N\setminus\N$. They do not need to hold even when $a,b\in\beta\N\setminus\N$. E.g., assume that $p,q,r$ are nonprincipal and $P=\{n\in\N\mid n \ \text{is a prime}\}\in p$. A direct computation shows that 
\[I_{1}=\{n\in\N\mid \exists p_{1},p_{2}\in P, a,b\in\N \ \text{s.t.}\ p_{1}<a<p_{2}<b \ \text{and} \ p_{1}^{a}p_{2}^{b}=n\}\in E_{1}(p,q)\odot E_{1}(p,r),\]
\[I_{2}=\{n\in\N\mid \exists p_{1},p_{2}\in P, a,b\in\N \ \text{s.t.}\ p_{1}>a>p_{2}>b \ \text{and} \ p_{1}^{a}p_{2}^{b}=n\}\in E_{2}(q,p)\odot E_{2}(r,p),\]

whilst 

\[J=\{n\in\N\mid \exists p\in P, a\in\N \ \text{s.t.}\  p^{a}=n\}\in E_{1}(p,q\oplus r)\cap E_{2}(q\oplus r, p),\]

and we conclude that $E_{1}(p,q)\odot E_{1}(p,r)\neq E_{1}(p,q\oplus r)$ and $E_{2}(q,p)\odot E_{2}(r,p)\neq E_{2}(q\oplus r,p)$ as $I_{1}\cap J=I_{2}\cap J=\emptyset.$\end{rem}

We now want to show that the operations $E_{1},E_{2}$ are related to the partition regularity of exponential triples, as well as of other exponential configurations. This should not come as a surprise: the main results in \cite{key-5} can be restated in the terminology introduced in Section \ref{prel} as follows:

\begin{thm}[\protect{\cite[Theorem 2.1, restated]{key-5}}] If $p$ is a Van der Waerden's witness then $2^{p}\odot p\models \{x,y,2^{x}y\}$.\end{thm}

From this fact, in \cite[Corollary 2.3]{key-5} the authors deduced the exponential Schur's Theorem by the $\log_{2}$-partition method mentioned in the introduction. Their proof actually shows the following fact:

\begin{thm}[\protect{\cite[Theorem 2.3, restated]{key-5}}]\label{DNRrestated} If $p$ is a Van der Waerden's witness then $2^{2^{p}\odot p}\models \{x,y,x^{y}\}$.\end{thm}

As it is well-known (see e.g. \cite[Theorem 14.5]{key-22}), any $p\in\overline{K(\beta\N,\odot)}$ is a Van der Waerden's witness; moreover, as $p\in\overline{K(\beta\N,\odot)}$, we have that $2^{p}\odot p\in\overline{K(\beta\N,\odot)}$. This simple observation, joint with Theorem \ref{DNRrestated} and Lemma \ref{old1}, allows to prove easily that several configurations are partition regular. In fact, let $p\in\overline{K(\beta\N,\odot)}$, and let $q=2^{p}$. Then:
\begin{itemize}
\item by Theorem \ref{DNRrestated}, $q\models x^{y}=z$; hence, repeated applications of Lemma \ref{old1} show that $q\models x_{1}^{x_{2}^{\udots^{x_{n}}}}=z$;
\item as $p\in\overline{K(\beta\N,\odot)}$, $p\models y_{1}+\ldots+y_{n}=z$, so $q\models y_{1}\cdot\ldots\cdot y_{n}=z$. By repeated applications of Lemma \ref{old1} we get that $q\models x_{1}^{x_{2}^{\udots^{x_{m}}}}=y_{1,1}^{\udots^{y_{1,n_{1}}}}\cdot\ldots\cdot y_{l,1}^{\udots^{y_{l,n_{l}}}}$;
\item by furtherly assuming that $p=p\odot p$, one gets that $q\models x^{\log_{2}(y)}=z$, hence e.g. $q\models x_{1}^{x_{2}\log_{2} x_{3}}=y_{1}\cdot y_{2}^{y_{3}}$;
\item if $p$ is also combinatorially rich\footnote{See e.g. \cite[Section 1.3]{advances} for a definition.}, then $p\models x=y+z^{2}$ (see e.g. \cite[Theorem 2.11]{advances}), hence $q\models x=yt^{\log_{2}(t)}$,
\end{itemize}

and so on, as we can always substitute a variable $x$ with a polynomial $P\left(y_{1},\ldots,y_{n}\right)$ whenever $q\models x=P\left(y_{1},\ldots,y_{n}\right)$. We refer to \cite[Theorem 2.11]{advances} for a long list of such admissible substitutions.

To conclude this section we study some properties of the set of exponential Schur's witnesses, namely

\[\mathrm{ESW}=\{p\in\beta\N\mid \forall A\in p\ \exists x,y\in A\  x^{y}\in A\}.\]

The set $\mathrm{ESW}$ has some nice closure property. To prove them, we will need the following strengthening of Di Nasso and Ragosta's results, whose proofs are almost identical to those of the results in \cite{key-5}. We add them here for completeness.

\begin{thm} Let $n\in\N$. If $p$ is a Van der Waerden's witness then $n^{p}\odot p\models\{x,y,n^{x}y\}$.\end{thm}

\begin{proof} Let $A\in n^{p}\odot p$. By definition, this means that $A^{\prime}=\{a\in\N\mid A/n^{a}\in p\}\in p$, where $A/n^{a}=\{m\in\N\mid n^{a}m\in A\}$. Pick $a\in A^{\prime}$. As $A^{\prime}\cap A/n^{a}\in p$, there are $b,c,b+n^{a}c\in A/n^{a}\cap A^{\prime}$. In particular, $A/n^{b}, A/n^{b+n^{a}c}\in p$. Let $d\in A/n^{b} \cap A/{n^{b+n^{a}c}}$. Letting $x=n^{a}c, y=n^{b}d$, we have that $x\in A$ as $c\in A/n^{a}$, $y\in A$ as $d\in A/n^{b}$, and $n^{x}y=n^{n^{a}c}n^{b}d=n^{b+n^{a}c}d\in A$ as $d\in A/n^{b+n^{a}c}$.\end{proof}

\begin{cor}\label{nuovoind} Let $n\in\N$. If $p$ is a Van der Waerden's witness then $n^{n^{p}\odot p}\in\mathrm{ESW}$.\end{cor}

\begin{proof} Let $A\in n^{n^{p}\odot p}$. Let $I_{A}\in \{m\mid n^{m}\in A\}\in n^{p}\odot p$. Let $a,b\in I_{A}$ be such that $n^{a}b\in I_{A}$. Then $x:=n^{b},y:=n^{a}$ are elements of $A$ such that $x^{y}=\left(n^{b}\right)^{\left(n^{a}\right)}=n^{bn^{a}}\in A$. \end{proof}

Our first result is that Corollary \ref{nuovoind} can be easily generalized as follows.

\begin{prop} Let $q\in\beta\N\setminus\{1\}$. If $p$ is a Van der Waerden's witness then $E_{1}\left(q,E_{1}(q,p)\odot p\right)\in \mathrm{ESW}$.
\end{prop}

\begin{proof} Let $A\in E_{1}\left(q,E_{1}(q,p)\odot p\right)$. By definition, 
\[\underbrace{\{n\in\N\mid \overbrace{\{m\in\N\mid \{a\in\N\mid \{b\in\N\mid n^{m^{a}b}\in A\}\in p\}\in p\}}^{J}\in q\}}_{I}\in q.\]
If $n\in I\cap J$ then $A\in n^{n^{p}\odot p}$, and we conclude by Corollary \ref{nuovoind}.\end{proof}

Similarly, a more restrictive condition on $p$ allows to prove the following:
\begin{prop}
\label{main2y} Let $p,q\in\bN\setminus\{1\}$. Assume that, for all $n>1$, $p\models \{x,y,xn^{y}\}$. Then $E_{1}\left(q,p\right)\in \mathrm{ESW}$.\end{prop}

\begin{proof} Let $A\in E_{1}(q,p)$. By definition, $I=\{n\in \N\mid \{m\in\N\mid n^{m}\in A\}\in p\}\in q$. Let $n\in I$ and $J_{n}=\{ m\in\N\mid n^{m}\in A\}$. Take $a,b\in J_{n}$ with $a,b,an^{b}\in J_{n}$. Therefore 
\[n^{a},n^{b},n^{an^{b}}=\left(n^{a}\right)^{\left(n^{b}\right)}\in A,\]
and we conclude by letting $x=n^{a},y=n^{b}$. \end{proof}

By \cite
[Theorem 4.40]{key-22}, any $q\in\overline{K(\beta\N,\odot)}$ consists solely of multiplicatively piecewise syndetic sets; hence, by Theorem \ref{GosPa}, any such $q$ satisfies that, for all $n>1$, $q\models \{x,y,xn^{y}\}$. Therefore, Theorem \ref{main2y} has the following immediate corollary:

\begin{cor} Let $q\in\overline{K(\beta\N,\odot)}$. Then for all $p\in\bN\setminus\{1\}$ we have that $E_{1}\left(p,q\right)\in \mathrm{ESW}$.\end{cor}

\begin{rem} If $\mathrm{ASW}=\{p\in\beta\N\mid \forall A\in p \ \exists a,b\in A \ a+b\in A\}$, Theorem \ref{Sahasr} tells us that $\mathrm{ASW}\cap \mathrm{ESW}=\emptyset$; in particular, if $E(\oplus)=\{p\in\bN\mid p\oplus p=p\}$ then
\[\overline{K(\beta\N,\odot)}\cap \mathrm{ESW}=\emptyset, \overline{E(\oplus)}\cap \mathrm{ESW}=\emptyset,\]
as $\overline{K(\beta\N,\odot)}\cup  \overline{E(\oplus)}\subseteq \mathrm{ASW}$. By Theorem \ref{main2y}.(1), we get that if $q\in\overline{K(\beta\N,\odot)}$ then $E_{1}\left(p,q\right)\notin \overline{K(\beta\N,\odot)}$; in particular, for $p,q\in\overline{K(\beta\N,\odot)}$, $E_{1}(p,q)\notin \overline{K(\beta\N,\odot)}$: the smallest bilateral multiplicative ideal of $\beta\N$ is very much not closed under the $E_{1}$ exponentiation.\end{rem}

\subsection{Non existence of exponential idempotents}\label{tr}

Having introduced two natural binary operations on $\beta\N$, it makes sense to ask if they admit idempotents, defined as follows.

\begin{defn} We say that $p\in\beta\N$ is $E_{1}$-idempotent (resp. $E_{2}$-idempotent) if $E_{1}(p,p)=p$ (resp. if $E_{2}(p,p)=p$).\end{defn}

 First of all, since $E_{1},E_{2}$ are not associative, we cannot approach the existence of idempotents problem via the usual method of applying Ellis' Theorem (see \cite[Theorem 2.5]{key-22} for details).  Moreover, since $1$ is the only principal $E_{1}-$idempotent, and the only principal $E_{2}$-idempotent as well, from now on we will use ``idempotent'' to actually mean ``nonprincipal idempotent''. Our starting  observation is the following:

\begin{prop}\label{stupid} Assume that there exist $p,q\in\beta\N\setminus\{1\}$ such that $p=p\odot E_{1}(q,p)$. Then $E_{1}(q,p)$ is $E_{1}$-idempotent. 

Analogously, if there exist $p,q\in\beta\N\setminus\{1\}$ such that $p=E_{2}(p,q)\odot p$ then $E_{2}(p,q)$ is $E_{2}$-idempotent.\end{prop}

\begin{proof} We just have to apply Proposition \ref{alg exp}.(5)-(6). In fact, if $p=p\odot E_{1}(q,p)$ then
\[E_{1}(q,p)=E_{1}(q,p\odot E_{1}(q,p))=E_{1}\left(E_{1}(q,p),E_{1}(q,p)\right)\]
and, if $p=E_{2}(p,q)\odot p$ then
\[E_{2}(p,q)=E_{2}\left(E_{2}(p,q)\odot p, q\right)=E_{2}\left(E_{2}(p,q),E_{2}(p,q)\right).\qedhere\]
\end{proof}

In particular, the existence of $n\in\N_{>1}$ and $p\in\beta\N\setminus\N$ with $p\odot n^{p}=p$ would force $n^{p}$ to be an $E_{1}$-idempotent; similarly, if $n^{p}\odot p=p$ then $n^{p}$ would be $E_{2}$-idempotent. 

However, idempotents cannot exist for $E_{1}$, and cannot exist for $E_{2}$ under certain additional assumptions, as we are now going to show. As for $E_{1}$, the non existence of idempotents follows from an even stronger fact, that we now prove.

\begin{thm}\label{noid} For all $p\in\beta\N, q\in\beta\N\setminus\N$ we have that\footnote{Notice, by contrast, that in Proposition \ref{alg exp} we proved that the apparently similar equation $E_{1}(p,q)=p$ has infinitely many solutions.} $E_{1}(p,q)\neq q$; in particular, there are no $E_{1}-$idempotents ultrafilters in $\beta\N\setminus\N$. \end{thm}

\begin{proof} By \cite[Lemma 23]{key-4}, there is a $4$-coloring $\N=A_{1}\cup\ldots\cup A_{4}$ of $\N$ such that there is no monochromatic pair $b,a^{b}$ with $\log_{2} a\leq b$ and $a,b> 1$. 

Now assume that $E_{1}(p,q)=q$; by definition, this means that for all $A\subseteq \N$ we have that 
\[A\in q\Leftrightarrow \{n\in\N\mid\{m\in\N\mid n^{m}\in A\}\in q\}\in p.\]
Now take $A_{i}\in\{A_{1},\ldots, A_{4}\}$ such that $A_{i}\in q$. In particular, $B=\{n\in\N\mid\{m\in\N\mid n^{m}\in A_{i}\}\in q\}\neq \emptyset$, so we can pick $n\in B$. Let $I_{n}:=\{m\in \N\mid n^{m}\in A_{i}\}$. As $q$ is non principal, $I_{n}$ must contain some $m>\log_{2}(n)$. But $m,n^{m}\in A_{i}$ with $\log_{2}(n)\leq m$, which is a contradiction. \end{proof}

Putting together Theorem \ref{noid} and Proposition \ref{stupid}, we get the following:

\begin{cor} For all $n>1\in\N$ and $p,q\in\beta\N\setminus\{0,1\}$ we have that $p\odot n^{p}\neq p$  and $p\odot E_{1}(q,p)\neq p$. \end{cor}

Naively, from Theorem \ref{noid} it would be natural to imagine that some similar argument could be used to prove the following conjecture:

\begin{con}\label{con} $E_{2}$-idempotent ultrafilters do not exist. \end{con}

As expected, a positive answer to the above conjecture would have several interesting combinatorial consequences: for example, any set belonging to such an ultrafilter would be exponentially-IP, in the following sense. 

\begin{thm}\label{exip} Let $p\in\beta\N$ such that $p = E_{2}(p, p)$. For each $A \in p$, there is a sequence $\langle x_{n}\rangle_{n=1}^{\infty}$ in $\N$ such that $\{x_{n}\mid n\in\N\}\cup\{x_{n+1}^{y}\mid n\in\N\ \text{and} \ y\in FP(\langle x_{t}\rangle_{t=1}^{n})\}\subseteq A$.\end{thm}

\begin{proof} For $A\subseteq \N$ and $x\in\N$, let $D(A, x) = \{y \in \N \mid y^{x}\in A\}$; for each $A\in p$, $\{x\in \N\mid D(A, x) \in p\}\in p$ by $E_{2}$-idempotency of $p$.

Let $A\in p$, let $A_{1} = A$, and let $B_{1} = \{x\in \N \mid D(A_{1}, x)\in p\}$. Pick $x_{1}\in A_{1}\cap B_{1}$ and let $A_{2} = A_{1}\cap B_{1} \cap D(A_{1}, x_{1})$. 

Now let $n\in \N\setminus\{1\}$ and assume that $A_{n}\in  p$. Let $B_{n} = \{x \in \N \mid D(A_{n}, x) \in p\}$, pick $x_{n}\in A_{n}\cap B_{n}$, and let $A_{n+1} = A_{n} \cap B_{n}\cap D(A_{n}, x_{n})$.

Having chosen $\langle x_{n}\rangle_{n=1}^{\infty}$, it is immediate that $\{x_{n}\mid  n \in \N\}\subseteq A$. Now let $n\in\N$. We
show by induction on $|F|$ that for $\emptyset\neq F\subseteq\{1,2,\dots,n\}$, if $m = \min F$ and $y =\prod_{t\in F}x_{t}$, then $x_{n+1}^{y}\in A_{m}$. If $F = \{m\}$, then $x_{n+1}\in A_{m+1}\subseteq D(A_{m}, x_{m})$ so $x_{n+1}^{x_m}\in A_{m}$. Now
assume that $|F| > 1$, let $G = F \setminus\{m\}$, let $r = \min G$, and let $z =\prod_{t\in G} x_{t}$. Then $x_{n+1}^{z}\in A_{r}\subseteq A_{m+1}\subseteq D(A_{m},x_{m})$ so $x_{n+1}^{zx_{m}}=\left(x_{n+1}^{z}\right)^{x_{m}}\in A_{m}$. \end{proof}

Unfortunately, we have not been able to prove Conjecture \ref{con}; we can, however, show a partial result towards its solution. The basic observation is the following: by Proposition \ref{alg exp}, if $p$ is $E_{2}$-idempotent, we have that 
\[p=E_{2}(p,p)=E_{2}(p,E_{2}(p,p))=E_{2}(p\odot p,p).\]
As, for all $n,m\in\N,p\in \beta\N\setminus\N$ one has that $E_{2}(n,p)=E_{2}(m,p)\Leftrightarrow n=m$, it becomes natural to ask if it is possible to have $p\in\beta\N\setminus\N$ such that $p=p\odot p=E_{2}(p,p)$; the answer is no.

\begin{thm} There is no ultrafilter $p\in\beta\N\setminus \N$ that is multiplicatively and $E_{2}$-idempotent. \end{thm}

\begin{proof} By contrast, let $p$ be such an ultrafilter. Let $\Omega:\N\rightarrow\N$ be the function that maps $n$ to the sum of the exponents in its prime factorization. In particular, $\widehat{\Omega}(p)=\widehat{\Omega}(E_{2}(p,p))=p\odot \widehat{\Omega}(p)$ and $\widehat{\Omega}(p)=\widehat{\Omega}(p\odot p)=\widehat{\Omega}(p)\oplus \widehat{\Omega}(p)$, therefore $\widehat{\Omega}(p)$ is additively idempotent and $p\odot \widehat{\Omega}(p)=\widehat{\Omega}(p)\oplus \widehat{\Omega}(p)$. If $\widehat{\Omega}(p)\in\N$, then by additive idempotency we would get that $\widehat{\Omega}(p)=0$, from which $p=1$, which is a contradiction as $p\in\beta\N\setminus\N$. So $\widehat{\Omega}(p)$ is additively idempotent and nonprincipal, hence $n\N\in \widehat{\Omega}(p)$ for all $n\in\N$. Thus we can apply Theorem \ref{HSsp} to deduce that $p\odot \widehat{\Omega}(p)\neq \widehat{\Omega}(p)\oplus \widehat{\Omega}(p)$, reaching the desired contradiction. \end{proof}

\section{Open questions}\label{questions}

In this section, we summarize some questions which we could not solve, as well as some possible directions we believe it could be interesting to pursue in the future.

The first question originates from the fact that, in \cite{key-4}, Sahasrabudhe proved the partition regularity of the configuration $\{x,y,xy,x^{y}\}$. We tried to recover a more direct proof of this fact, along the lines of what we developed in Section \ref{extexp}, but we failed. Our guess is that, by choosing $p$ in a reasonable way, the partition regularity of $\{x,y,xy,x^y\}$ should be witnessed by $2^{p}$. We conjecture that a family of ultrafilters $p$ that should make our claim true is the family of those $ p\in E \left( K(\bN,\odot ) \right) $ such that each member of $p$ is also additive central. Summarizing, our first question is the following:
\begin{question}
    Let $ p\in E \left( K(\bN,\odot ) \right) $ be such that each member of $p$ is also additively central. Is it true or not that each $A\in p$ contains a configuration of the form $\{m,n,m+n,m2^n\}$?
\end{question}
Notice that an affirmative answer to the above question would immediately entail, by using the same methods adopted in this paper, that each member of $2^p$ contains patterns of the form $\{x,y,xy,x^y\}$. 

The following question regards several technical properties of the exponential operations $E_{1}, E_{2}$. We proved the existence and non-existence of ultrafilters satisfying certain equalities; for some other similar ones, we did not manage to find a definitive answer.
 \begin{question}
Prove that there exist $p,q,r\in \bN\setminus \N$ that solve any of the following equations 
\begin{enumerate}
\item $E_{1}(p,q)=E_{1}(q,p)$;
\item $E_{2}(p,q)=E_{2}(q,p)$;
\item $E_{1}(p,q)=E_{2}(q,p)$; 
\item $E_{1}(q,p)\odot p=p$;
\item $E_{1}(q,p)\odot q=q$;
\item $E_{1}(r,p)\odot E_{1}(r,q)= E_{1}(r,p\oplus q)$;
\item $E_{2}(q,p)\odot p=p$;
\item $E_{2}(q,p)\odot q=q$;
\item $E_{2}(p,r)\odot E_{2}(q,r)= E_{2}(p\oplus q,r)$.
\end{enumerate}
\end{question}

One particular case of equality was discussed in Example \ref{equalE1E2}, from which the following question arises:

\begin{question} For which $p,q\in\beta\N\setminus\N$ does the equality $E_{1}(p,q)=E_{2}(p,q)$ hold? Does it exist $p\in\beta\N\setminus\N$ such that $E_{1}(p,p)=E_{2}(p,p)$?\end{question}

Of course, as discussed in the paper, more relevant questions are the following:

\begin{question} Does it exist $p$ that is $E_{2}$-idempotent?\end{question}

Finally, as a general direction connected to all the above questions, we believe that it would be interesting to further study the properties of exponentially rich ultrafilters; of course, this rather vague idea could be made precise in several different ways. We believe it would be interesting to start by studying the following:

\begin{question} What are the algebraical properties of $\mathrm{ESW}$? \end{question}

{\bfseries Acknowledgements:} We thank the referee for his/her careful reading of the paper, which led to several improvements; in particular, we are thankful for the suggestion of isolating Lemma \ref{lem:basicproplogexp} to simplify the proofs of Proposition \ref{alg exp}.(2)-(3), for proposing an explicit proof of the properties of $\widehat{H}$ used in Case 2 of Proposition \ref{neqr}, and for indicating Theorem \ref{exip} as a striking example of the eventual consequences of the existence of $E_{2}$-idempotents.

\end{document}